\documentclass[12pt]{amsart}
\usepackage{amssymb,latexsym}
\usepackage{amsfonts}
\usepackage{amsmath}
\usepackage[colorlinks,linkcolor=blue,anchorcolor=blue,citecolor=blue]{hyperref}
\usepackage{algorithm}
\usepackage{enumerate}
\usepackage{algpseudocode}
\usepackage{verbatim}
\usepackage{graphicx}

\newcommand\Z{{\mathbb Z}}

\newtheorem{theorem}{Theorem}[section]
\newtheorem{lemma}[theorem]{Lemma}

\newtheorem{corollary}[theorem]{Corollary}

\theoremstyle{definition}

\newtheorem{remark}[theorem]{Remark}

\numberwithin{equation}{section}

\newcommand\Res{\mathrm{Res}}

\begin{document}

\title[Zsigmondy's theorem and primitive divisors]{Zsigmondy's theorem and primitive divisors of the Lucas and Lehmer sequences in polynomial rings}

\author{Min Sha}
\address{School of Mathematical Sciences, South China Normal University, Guangzhou, 510631, China}
\email{shamin@scnu.edu.cn}



\begin{abstract}
In this paper, we obtain analogues of Zsigmondy's theorem and the primitive divisor results for the Lucas and Lehmer sequences 
in polynomial rings of several variables. 
\end{abstract}

\maketitle

\section{Introduction}   \label{sec:int}

\subsection{Background and motivation}
Given a sequence $(a_n)_{n \ge 1}$ of the rational integers $\Z$, a prime divisor of a term $a_n$ is called \textit{primitive} if it divides no earlier term. 
The sequence is a \textit{divisibility sequence} if $a_m \mid a_n$ whenever $m \mid n$, 
and it is a \textit{strong divisibility sequence} if $\gcd(a_m, a_n) = a_d$ with $d = \gcd(m,n)$ for any positive integers $m, n$. 
These notions apply to any sequence of a unique factorization domain. 

It is a classical and still very active topic in number theory to study primitive prime divisors of an integer sequence. 
The classical Zsigmondy theorem \cite{Zsig} in 1892, extending earlier work of Bang \cite{Bang} in the case $b=1$, says that every term beyond the sixth in the sequence 
$(a^n - b^n)_{n \ge 1}$ has a primitive prime divisor, where $a,b$ are  positive coprime integers. 
This theorem was independently rediscovered by Birkhoff and Vandiver \cite{BV}. 
Results of this form are often useful in group theory and in the theory of recurrence sequences (see \cite[Section 6.3]{EPSW} for a discussion and references).

In 1913, Carmichael \cite{Car} showed that each term of the Lucas sequence $((a^n - b^n)/(a-b))_{n \ge 1}$ beyond the twelfth has a primitive prime divisor, 
where $a, b$ are real algebraic integers such that $a/b$ is not a root of unity, and $a+b$ and $ab$ are coprime integers in $\Z$. 
In 1955, Ward \cite{Ward} obtained a similar result for the Lehmer sequence $(s_n)_{n \ge 1}$ with $s_n = (a^n - b^n)/(a-b)$ for odd $n$ 
and $s_n = (a^n - b^n)/(a^2-b^2)$ for even $n$, where $a,b$ are real, and $(a+b)^2$ and $ab$ are coprime integers in $\Z$.
All these results, including Zsigmondy's theorem, were extended to any number field (that is, $a, b$ do not need to be real) by Schinzel \cite{Sch}  
in an effective but not explicit manner (see \cite{PS} for an earlier work), 
which was first made explicitly by Stewart \cite{Stewart}. 
Furthermore, in 2001, Bilu, Hanrot and Voutier \cite{BHV} listed all the Lucas and Lehmer numbers without primitive prime divisor. 

So far, the above classical results have various extensions in different settings. 
For example, the extensions to elliptic divisibility sequence \cite{EMW,Sil},  to dynamical sequences \cite{IS,Rice}, 
to function fields defined over number fields \cite{IMSSS}, to Drinfeld modules \cite{Bam,Quan,ZJ}. 
Recently, Flatters and Ward \cite{FW} found an analogue of Zsigmondy's theorem for a polynomial sequence $(f^n - g^n)_{n \ge 1}$, 
where $f,g$ are two coprime polynomials in a polynomial ring $K[X]$ ($K$ is a field). 

In this paper, we want to establish analogues of Zsigmondy's theorem and the primitive divisor results for the Lucas and Lehmer sequences 
in polynomial rings of several variables. 
The approach is essentially the same as in \cite{FW}. 
It in fact follows the classical one with some modifications needed to avoid terms in the sequence where the Frobenius automorphism precludes primitive divisors. 
However, for analogues of polynomial Lucas and Lehmer sequences, it indeed needs some more considerations. 

Throughout the paper, let $K$ be a field, and $R = K[X_1, \ldots, X_r]$ the ring of polynomials in varibales $X_1, \ldots, X_r$. 
Let $p$ be the characteristic of $K$. Note that $R$ is a unique factorization domain. 
Besides, a \textit{prime divisor} of a polynomial $h$ in $R$ means a monic irreducible polynomial in $R$ dividing $h$. 

We state the main results in the rest of Section~\ref{sec:int}, and then prove them later on.

\subsection{Main results} 

Let $\lambda, \eta$ be non-zero algebraic elements over $R$ such that $\lambda / \eta$ is not a root of unity. 
Assume that $(\lambda + \eta)^2$ and $\lambda\eta$ are non-zero coprime polynomials in $R$ which are not both in $K$. 
Define the \textit{Lehmer sequence} of $R$: 
\begin{equation*}
U_n = 
\left\{\begin{array}{ll}
\frac{\lambda^n - \eta^n}{\lambda - \eta} & \textrm{if $n$ is odd,}\\
\\
\frac{\lambda^n - \eta^n}{\lambda^2 - \eta^2} & \textrm{if $n$ is even.}\end{array}\right.
\end{equation*}
We remark that the Lehmer sequence $(U_n)_{n \ge 1}$ satisfies the following recurrence relation over $R$: 
$$
U_{n+4} = (\lambda^2 + \eta^2) U_{n+2} - \lambda^2\eta^2 U_n, \quad n = 1, 2, \ldots. 
$$

The following two theorems are about the strong divisibility property and the primitive prime divisors of the sequence $(U_n)_{n \ge 1}$, respectively. 

\begin{theorem}  \label{thm:strong3}
The sequence $(U_n)_{n \ge 1}$ is a strong divisibility sequence. 
\end{theorem}

\begin{theorem}  \label{thm:primitive3}
Suppose the characteristic $p > 0$ and let $U^\prime$ be the sequence obtained from $(U_n)_{n \ge 1}$ by deleting the terms $U_n$ with $p \mid n$, 
then each term of $U^\prime$ beyond the second has a primitive prime divisor. 
If $p = 0$, then each term of $(U_n)_{n \ge 1}$ beyond the second has a primitive prime divisor. 
\end{theorem}

Applying Theorems~\ref{thm:strong3} and \ref{thm:primitive3}, 
we can obtain the strong divisibility property and the primitive divisor result for polynomial Lucas sequences. 

Let $\alpha, \beta$ be non-zero algebraic elements over $R$ such that the quotient $\alpha/\beta$ is not a root of unity. 
Assume that $\alpha + \beta$ and $\alpha \beta$ are coprime polynomials in $R$ which are not both in $K$. 
Define the \textit{Lucas sequence} of $R$: 
$$
L_n = \frac{\alpha^n - \beta^n}{\alpha - \beta}, \quad n =1,2, \ldots. 
$$
We remark that the Lucas sequence $(L_n)_{n \ge 1}$ satisfies the following recurrence relation over $R$: 
$$
L_{n+2} = (\alpha + \beta) L_{n+1} - \alpha\beta L_n, \quad n = 1, 2, \ldots. 
$$

\begin{theorem}  \label{thm:strong2}
The sequence $(L_n)_{n \ge 1}$ is a strong divisibility sequence. 
\end{theorem}

\begin{theorem}  \label{thm:primitive2}
Suppose the characteristic $p > 0$ and let $L^\prime$ be the sequence obtained from $(L_n)_{n \ge 1}$ by deleting the terms $L_n$ with $p \mid n$, 
then each term of $L^\prime$ beyond the second has a primitive prime divisor. 
If $p = 0$, then each term of $(L_n)_{n \ge 1}$ beyond the second has a primitive prime divisor. 
\end{theorem}

Theorem~\ref{thm:primitive2} in fact implies an analogue of Zsigmondy's theorem in $R$. 
Let $f, g$ be non-zero coprime polynomials in $R$ such that $f$ and $g$ are not both in $K$ and the quotient $f/g$ is not a root of unity. 
Define the sequence of $R$: 
$$
F_n = f^n - g^n, \quad n =1,2, \ldots. 
$$

\begin{theorem}  \label{thm:strong1}
The sequence $(F_n)_{n \ge 1}$ is a strong divisibility sequence. 
\end{theorem}

\begin{theorem}  \label{thm:primitive1}
Suppose the characteristic $p > 0$ and let $F^\prime$ be the sequence obtained from $(F_n)_{n \ge 1}$ by deleting the terms $F_n$ with $p \mid n$, 
then each term of $F^\prime$ beyond the second has a primitive prime divisor. 
If $p = 0$, then each term of $(F_n)_{n \ge 1}$ beyond the second has a primitive prime divisor. 
\end{theorem}

For $f, g$ as the above, define the sequence:  
$$
S_n = f^n + g^n, \quad n = 1,2, \ldots. 
$$
If $p \ne 2$, then $(S_n)_{n \ge 1} \ne (F_n)_{n \ge 1}$. 
Note that $F_{2n} = F_n S_n$, which implies that each primitive prime divisor of $F_{2n}$ comes from $S_n$. 
Then, the following corollary is a direct consequence of Theorem~\ref{thm:primitive1}. 

\begin{corollary}  \label{cor:primitive1}
Suppose $p > 0$, and let $S^\prime$ be the sequence obtained from $(S_n)_{n \ge 1}$ by deleting the terms $S_n$ with $p \mid n$, 
then each term of $S^\prime$ beyond the second has a primitive prime divisor. 
If $p = 0$, then each term of $(S_n)_{n \ge 1}$ beyond the second has a primitive prime divisor. 
\end{corollary}

\section{Preliminaries}

Although the sequences we consider are defined over $R$, we prefer to establish some results in a more general setting. 
Throughout this section, let $D$ be a unique factorization domain. 

Recall that the resultant of two homogeneous polynomials in variables $X$ and $Y$ is defined to the determinant of their Sylvester matrix. 
Some basic properties of this resultant are listed  in the following lemma; see \cite[Proposition 2.3]{Silverman}. 

\begin{lemma}   \label{lem:Res}
For two non-constant homogeneous polynomials defined over a field 
$$
A(X,Y) = a_0 X^m + a_1 X^{m-1}Y + \ldots + a_m Y^m = a_0 \prod_{i=1}^{m} (X - \alpha_i Y)
$$  
and 
$$
B(X,Y) = b_0 X^n + b_1 X^{n-1}Y + \ldots + b_n Y^n = b_0 \prod_{j=1}^{n} (X - \beta_j Y), 
$$
their resultant is 
$$
\Res(A,B) = a_0^n b_0^m \prod_{i=1}^{m} \prod_{j=1}^{n} (\alpha_i - \beta_j) \in \Z[a_0, \ldots, a_m, b_0, \ldots, b_n]. 
$$
Moreover, there exist $G_1, H_1, G_2, H_2 \in \Z[a_0, \ldots, a_m, b_0, \ldots, b_n][X,Y]$ homogeneous in $X$ and $Y$ such that 
\begin{align*}
& G_1A + H_1 B = \Res(A,B)X^{m+n-1}, \\
& G_2A + H_2 B = \Res(A,B)Y^{m+n-1}.
\end{align*}
\end{lemma}

For any integer $n \ge 1$, the $n$-th homogeneous cyclotomic polynomial is defined by 
$$
\Phi_n (X,Y) = \prod_{k=1, \, \gcd(k,n)=1}^{n} (X - \zeta_n^k Y) \in \Z[X,Y], 
$$
where $\zeta_n$ is a primitive $n$-th root of unity, and we also define the polynomial
$$
P_n(X,Y) = \frac{X^n - Y^n}{X-Y} = \sum_{k=0}^{n-1} X^{n-1-k}Y^k = \prod_{k=1}^{n-1} (X - \zeta_n^k Y).
$$ 
Then, it is easy to see that 
$$
X^n - Y^n = \prod_{d \mid n} \Phi_d (X,Y), \qquad P_n (X,Y) = \prod_{d \mid n, \, d \ge 2} \Phi_d (X,Y). 
$$

The following result is \cite[Lemma 2.4]{FW} about the resultant of $P_m(X,Y)$ and $P_n(X,Y)$. 

\begin{lemma}  \label{lem:Res2}
For any positive coprime  integers $m$ and $n$, we have $\Res(P_m, P_n) = \pm 1$. 
\end{lemma}

We now want to establish some results about coprime elements in $D$. 
First, we prove a general result. 

\begin{lemma}  \label{lem:coprime}
Let $a, b$ be algebraic elements over $D$. 
Assume that $(a+b)^2$ and $ab$ are coprime elements in $D$. 
Let $A(X,Y), B(X,Y) \in \Z[X,Y]$ be non-constant homogeneous polynomials with resultant $\Res(A,B) = \pm 1$. 
Assume that both $A(a,b)$ and $B(a,b)$ are in $D$. 
Then, $A(a,b)$ and $B(a,b)$ are coprime in $D$. 
\end{lemma}

\begin{proof}
Let $m = \deg A$ and $n = \deg B$. 
By assumption, $a^2 + b^2$ and $ab$ are also coprime in $D$. 
Note that for any  integer $k \ge 1$, $a^{2k} + b^{2k}$ is in $D$. 
Using Lemma~\ref{lem:Res} and noticing $\Res(A,B) = \pm 1$, we obtain that there exist $u_1, w_1, u_2, w_2 \in \Z[a, b]$ such that 
$$
u_1 A(a,b) + w_1B(a,b) = a^{2(m+n-1)} + b^{2(m+n-1)} \in D
$$
and 
$$
u_2 A(a,b) + w_2 B(a,b) = a^{2(m+n-1)}b^{2(m+n-1)} \in D. 
$$
Note that $u_1, w_1, u_2, w_2$ might be not in $D$. 

By contradiction, suppose that $A(a,b)$ and $B(a,b)$ are not coprime in $D$. 
Then, there is a prime element $\pi \in D$ such that $\pi \mid A(a,b)$ and $\pi \mid B(a,b)$ in $D$. 

By the above discussion, we obtain that both 
$$
\frac{u_1 A(a,b) + w_1 B(a,b)}{\pi}, \quad \frac{u_2 A(a,b) + w_2 B(a,b)}{\pi}
$$
are in the fraction field of $D$ and integral over $D$ 
(because $a,b$ are integral over $D$ satisfying the equation $X^4 - (a^2 + b^2)X^2 + a^2b^2 = 0$). 
Note that $D$ is an integrally closed domain, so these two quotients are both in $D$. 
Hence, we have $\pi \mid a^{2(m+n-1)} + b^{2(m+n-1)}$ and $\pi \mid a^{2(m+n-1)}b^{2(m+n-1)}$ in $D$. 
So, $\pi \mid ab$ in $D$. 

Let $k = m+n-1$. We have known that $\pi \mid a^{2k} + b^{2k}$ and $\pi \mid ab$ in $D$. 
Consider 
\begin{align*}
(a^2 + b^2)^k & = a^{2k} + b^{2k} + \sum_{i=1}^{k-1}\binom{k}{i}(a^2)^i (b^2)^{k-i}  \\
& =  a^{2k} + b^{2k} + a^2b^2 \sum_{i=1}^{k-1}\binom{k}{i}(a^2)^{i-1} (b^2)^{k-1-i}.
\end{align*}
Note that $\sum_{i=1}^{k-1}\binom{k}{i}(a^2)^{i-1} (b^2)^{k-1-i}$ is also in $D$, because it is symmetric in $a^2$ and $b^2$. 
Hence, we have $\pi \mid (a^2 + b^2)^{k}$, and then $\pi \mid a^2 + b^2$,  so $\pi \mid (a + b)^2$ in $D$ (because $\pi \mid ab$). 
This leads to a contradiction with the assumption that $(a + b)^2$ and $ab$ are coprime in $D$. 
Therefore, $A(a,b)$ and $B(a,b)$ are coprime in $D$. 
\end{proof}

Based on Lemma~\ref{lem:coprime}, we can derive several results about coprime elements in $R$ in the sequel. 

\begin{lemma}  \label{lem:PmPn2}
Let $a, b$ be two algebraic elements over $D$. 
Assume that $a+b$ and $ab$ are coprime elements in $D$. 
Then, for any positive coprime integers $m, n$, $P_m(a,b)$ and $P_n(a,b)$ are coprime in $D$. 
\end{lemma}

\begin{proof}
Without loss of generality, we can assume $m \ge 2, n \ge 2$.
Clearly, both $P_m(a,b)$ and $P_n(a,b)$ are in $D$. Because both $P_m(X,Y)$ and $P_n(X,Y)$ are symmetric in $X$ and $Y$. 

By assumption, $(a+b)^2$ and $ab$ are coprime elements in $D$. 
Then, the desired result follows from Lemmas~\ref{lem:Res2} and \ref{lem:coprime}. 
\end{proof}

\begin{lemma}  \label{lem:PmPn-odd}
Let $a, b$ be defined as in Lemma~\ref{lem:coprime}. 
Let $m,n$ be two positive coprime  integers such that both $m$ and $n$ are odd. 
Then, $P_m(a,b)$ and $P_n(a,b)$ are coprime in $D$. 
\end{lemma}

\begin{proof}
Without loss of generality, we can assume $m \ge 3, n \ge 3$. 
Since both $(a+b)^2$ and $ab$ are in $D$, we have $a^2 + b^2 \in D$. 
Moreover, we have that for any integer $k \ge 1$, $a^{2k} + b^{2k} \in D$.

Note that $P_m(X,Y)$ is homogeneous of even degree $m-1$ and symmetric in $X$ and $Y$. 
So, if $X^iY^j$ is a term in $P_m(X,Y)$, then $X^jY^i$ is also a term in $P_m(X,Y)$, and then assuming $i \le j$, we have
$$
a^i b^j + a^j b^i = (ab)^i (a^{j-i} + b^{j-i}) \in D, 
$$
where we use the fact that $j - i$ is even (because $i + j = m-1$ is even). 
Hence, we have that $P_m(a,b)$ is in $D$. 
Similarly, $P_n(a,b)$ is also in $D$. 
Now, the desired result follows directly from Lemmas~\ref{lem:Res2} and \ref{lem:coprime}. 
\end{proof}

\begin{lemma}  \label{lem:PmPn-mix}
Let $a, b$ be defined as in Lemma~\ref{lem:coprime}. 
Let $m,n$ be two positive coprime  integers such that  $m$ is odd and $n$ is even. 
Then, $P_m(a,b)$ and $P_n(a,b)/(a + b)$ are coprime in $D$. 
\end{lemma}

\begin{proof}
Without loss of generality, we can assume $m \ge 3, n \ge 4$ (because $P_2(a,b)/(a + b)=1$). 
Since $m$ is odd, as in the proof of Lemma~\ref{lem:PmPn-odd} $P_m(a,b)$ is in $D$. 
For any odd integer $k \ge 1$, note that
$$
\frac{a^k + b^k}{a+b} = a^{k-1} - a^{k-2}b + \ldots - ab^{k-2} + b^{k-1}
$$
is homogeneous of even degree $k-1$ and is symmetric in $a$ and $b$, so it is in $D$. 
Hence, for even $n$, since 
\begin{align*}
\frac{P_n(a,b)}{a + b} = \frac{a^{n-1} + b^{n-1}}{a+b}  + ab \cdot \frac{a^{n-3} + b^{n-3}}{a+b} + \ldots 
 + a^{\frac{n-2}{2}}b^{\frac{n-2}{2}} \cdot \frac{a + b}{a+b}, 
\end{align*}
we have that $P_n(a,b)/(a + b)$ is in $D$. 

Denote $T_n(X,Y) = P_n(X,Y)/(X+Y)$, which can be viewed as a polynomial over $\Z$.
Using Lemma~\ref{lem:Res} and applying the same arguments as in the proof of \cite[Lemma 2.4]{FW}, 
we obtain that the resultant of $P_m(X,Y)$ and $T_n(X,Y)$ is equal to $\pm 1$. 

Hence, by Lemma~\ref{lem:coprime}, $P_m(a,b)$ and $T_n(a,b)$ are coprime in $D$.
\end{proof}

\begin{lemma}  \label{lem:Pmn}
Let $a, b$ be defined as in Lemma~\ref{lem:coprime}. 
Let $m, n$ be two positive integers such that both $m$ and $n$ are odd. 
Then, $P_m(a^n,b^n)$ and $(a^n + b^n)/(a + b)$ are coprime in $D$. 
\end{lemma}

\begin{proof}
Without loss of generality, we can assume $m \ge 3, n \ge 3$. 
As before, since $m$ and $n$ are odd, both $P_m(a^n,b^n)$ and $(a^n + b^n)/(a + b)$ are indeed in $D$.

Define
$$
V_m(X,Y) =P_m(X^n, Y^n), \qquad  W_n(X,Y) = \frac{X^n + Y^n}{X+Y}. 
$$
Both $V_m$ and $W_n$ can be viewed as polynomials over $\Z$. 
So, we first compute their resultant over $\Z$. 
Note that 
$$
V_m(X,Y) = \prod_{i=1}^{m-1}(X^n - \zeta_m^i Y^n) = \prod_{i=1}^{m-1}\prod_{j=1}^{n} (X - \zeta_n^j \zeta_{mn}^i  Y), 
$$
and 
$$
W_n(X,Y) =  \frac{X^n - (-Y)^n}{X+Y} = \prod_{k=1}^{n-1}(X + \zeta_n^k Y).
$$
By Lemma~\ref{lem:Res}, the resultant 
$$
\Res(V_m, W_n) = \prod_{i=1}^{m-1}\prod_{j=1}^{n} \prod_{k=1}^{n-1} ( \zeta_{mn}^i \zeta_n^j + \zeta_n^k) \in \Z. 
$$ 
For each factor $\zeta_{mn}^i \zeta_n^j + \zeta_n^k$ in the resultant, we have $\zeta_{mn}^i \zeta_n^j \ne \zeta_n^k$ 
(because otherwise we would have $\zeta_{mn}^i \zeta_{mn}^{mj} = \zeta_{mn}^{mk}$, and then $m \mid i$, but $1 \le i \le m-1$), 
and so $\zeta_{mn}^i \zeta_n^j + \zeta_n^k = \zeta_{m^\prime}(1 -\zeta_2\zeta_{n^\prime})$ 
for some odd integer $m^\prime$ and odd integer $n^\prime \ge 3$ (noticing both $m,n$ are odd), 
and thus it is a unit by \cite[Proposition 2.8]{Was}. 
Hence, $\Res(V_m, W_n)$ is a unit in $\Z$, that is, $\Res(V_m, W_n) = \pm 1$. 

Therefore, as polynomials over $D$, we also have $\Res(V_m, W_n) = \pm 1$. 
Then, by Lemma~\ref{lem:coprime}, $V_m(a,b)$ and $W_n(a,b)$ are coprime in $D$.
\end{proof}

\begin{lemma}  \label{lem:abn}
Let $a, b$ be defined as in Lemma~\ref{lem:coprime}. 
Then, for any odd integer $n \ge 1$,  $(a^n - b^n)/(a-b) $ and $(a+b)^2$  are coprime in $D$. 
\end{lemma}

\begin{proof}
Without loss of generality, we fix an odd integer $n \ge 3$. 
As before, both $(a^n - b^n)/(a-b) $ and $(a+b)^2$ are in $D$. 

As in the proof of Lemma~\ref{lem:Pmn}, we deduce that the resultant of the homogeneous polynomials 
$(X^n - Y^n)/(X-Y) $ and $(X+Y)^2$ is equal to $\pm 1$. 
Hence, using Lemma~\ref{lem:coprime}, we obtain that $(a^n - b^n)/(a-b) $ and $(a+b)^2$  are coprime in $D$.  
\end{proof}

\section{Proofs of Theorems~\ref{thm:strong3} and \ref{thm:primitive3}}  

We need to make one more preparation. 

Recall that $p$ is the characteristic of the field $K$. 
As usual, denote by $v_\pi(h)$ the maximal power to which an irreducible polynomial $\pi$ divides $h \in R$. 

Let $M$ be the fraction field of $R$. 
By assumption, $M(\lambda)$ is a field extension over $M$ having degree at most four. Note that $\eta \in M(\lambda)$. 
For any irreducible polynomial $\pi \in R$, as usual $v_\pi$ induces a valuation of $M$. 
It is well-known that the valuation $v_\pi$ in $M$ can be extended to the field $M(\lambda)$; see, for instance, \cite[Theorem 3.1.2]{EP}. 
Without confusion, we still denote by $v_\pi$ the corresponding extension of valuation in $M(\lambda)$. 

\begin{lemma}  \label{lem:vU}
Let $\pi \in R$ be an irreducible polynomial dividing $U_n$ for some $n \ge 3$. 
Then, for any $m \ge 1$ with $p \nmid m$ $($including the case $p=0)$,  
we have $v_\pi (U_{mn}) = v_\pi (U_n)$. 
\end{lemma}

\begin{proof}
First, since $\lambda, \eta$ are both integral over the ring $R$, 
we have that $v_\pi(\lambda) \ge 0$ and  $v_\pi(\eta) \ge 0$. 

Suppose that $v_\pi(\eta) > 0$. 
Note that we have either $\lambda^n = \eta^n + (\lambda - \eta)U_n$, or $\lambda^n = \eta^n + (\lambda^2 - \eta^2)U_n$. 
Then, since $v_\pi(\eta) > 0$ and $v_\pi(U_n) > 0$, we have $v_\pi(\lambda^n) > 0$. 
So, $v_\pi (\lambda) > 0$. 
Thus, 
$$
v_\pi(\lambda + \eta) >0, \qquad v_\pi(\lambda\eta) > 0, 
$$
which contradicts the assumption that $(\lambda + \eta)^2$ and $\lambda\eta$ are coprime in $R$. 
Hence, we must have $v_\pi(\eta) = 0$. 
Similarly, we must have $v_\pi(\lambda) = 0$. 

Assume that $n$ is odd. 
Then, $U_n = (\lambda^n - \eta^n)/(\lambda - \eta)$. 
 So, we have 
$$
\lambda^n = \eta^n +  (\lambda - \eta)U_n. 
$$
Then, we obtain
$$
\lambda^{mn} = \big( \eta^n + (\lambda - \eta)U_n \big)^m 
= \eta^{mn} + \sum_{i=1}^{m} \binom{m}{i} (\lambda - \eta)^i U_n^i  \eta^{n(m-i)}. 
$$
So 
$$
\frac{\lambda^{mn} - \eta^{mn}}{\lambda - \eta} = m\eta^{n(m-1)} U_n + \sum_{i=2}^{m} \binom{m}{i} (\lambda - \eta)^{i-1} \eta^{n(m-i)}  U_n^i.
$$
Hence, we obtain that for odd $m$
$$
U_{mn} = m\eta^{n(m-1)}U_n + \sum_{i=2}^{m}\binom{m}{i} (\lambda - \eta)^{i-1}\eta^{n(m-i)} U_n^{i},
$$
and for even $m$ 
$$
(\lambda + \eta)U_{mn} = m\eta^{n(m-1)}U_n + \sum_{i=2}^{m}\binom{m}{i} (\lambda - \eta)^{i-1}\eta^{n(m-i)} U_n^{i}.
$$
We also note that since $n$ is odd and $v_\pi (U_n) > 0$, by Lemma~\ref{lem:abn} we have $v_\pi (\lambda + \eta) = 0$. 
Then, the desired result follows.

Finally, assume that $n$ is even. 
Then, as the above, for any integer $m \ge 1$ we obtain 
$$
U_{mn} = m\eta^{n(m-1)}U_n + \sum_{i=2}^{m}\binom{m}{i} (\lambda^2 - \eta^2)^{i-1}\eta^{n(m-i)} U_n^{i}.
$$
The desired result now follows.
\end{proof}

Now, we are ready to prove the theorems. 

\begin{proof}[Proof of Theorem~\ref{thm:strong3}]
Let $d = \gcd(m,n)$. 

First, we assume that both $m$ and $n$ are even. Then, $d$ is also even. 
By definition, we obtain 
$$
U_m = U_d P_{m/d}(\lambda^d, \eta^d), \quad U_n = U_d P_{n/d}(\lambda^d, \eta^d). 
$$
By assumption, it is easy to see that $\lambda^d + \eta^d$ and $\lambda^d \eta^d$ are coprime in $R$ 
(as in the last paragraph of the proof of Lemma~\ref{lem:coprime}).   
Hence, by Lemma~\ref{lem:PmPn2}, we know that $P_{m/d}(\lambda^d, \eta^d)$ and $P_{n/d}(\lambda^d, \eta^d)$ are coprime in $R$, 
and so we have $\gcd(U_m, U_n) = U_d$ in this case.

Now, we assume that both $m$ and $n$ are odd. Then, $d$ is also odd. 
By definition, we have 
$$
U_m = U_d P_{m/d}(\lambda^d, \eta^d), \quad U_n = U_d P_{n/d}(\lambda^d, \eta^d). 
$$
We also note that $(\lambda^d + \eta^d)^2$ and $\lambda^d \eta^d$ are coprime in $R$.  
Then, by Lemma~\ref{lem:PmPn-odd} we know that $P_{m/d}(\lambda^d, \eta^d)$ and $P_{n/d}(\lambda^d, \eta^d)$ are coprime in $R$, 
and so we have $\gcd(U_m, U_n) = U_d$.

Finally, when $m$ and $n$ do not have the same parity, without loss of generality, we assume that  $m$ is odd and $n$ is even. 
Then, $d$ is odd. By definition, we have 
$$
U_m = U_d P_{m/d}(\lambda^d, \eta^d),
$$
and 
$$
 U_n = U_d \cdot \frac{P_{n/d}(\lambda^d, \eta^d)}{\lambda^d + \eta^d} \cdot \frac{\lambda^d + \eta^d}{\lambda + \eta}. 
$$
Then, by Lemma~\ref{lem:PmPn-mix} we know that $P_{m/d}(\lambda^d, \eta^d)$ and $P_{n/d}(\lambda^d, \eta^d)/(\lambda^d + \eta^d)$ are coprime in $R$. 
Besides, by Lemma~\ref{lem:Pmn} we obtain that $P_{m/d}(\lambda^d, \eta^d)$ and $(\lambda^d + \eta^d)/(\lambda + \eta)$ are coprime in $R$. 
Hence, we have $\gcd(U_m, U_n) = U_d$.
This completes the proof. 
\end{proof}

\begin{proof}[Proof of Theorem~\ref{thm:primitive3}] 
As in \cite{Ward}, we define the sequence $(Q_n)_{n \ge 1}$ of polynomials by $Q_1 = 1, Q_2 = 1$, and 
$$
Q_n(X,Y) = \Phi_n(X,Y), \quad n = 3, 4, \ldots. 
$$
Then, it is easy to see that for any integer $n \ge 1$ we have 
$$
U_n = \prod_{d \mid n} Q_d(\lambda, \eta). 
$$
By the M{\"o}bius inversion, we have 
$$
Q_n(\lambda, \eta) = \prod_{d\mid n} U_d^{\mu(n/d)}. 
$$
So, for any irreducible polynomial $\pi$ in $R$ we have 
$$
v_\pi(Q_n(\lambda, \eta)) = \sum_{d \mid n} \mu(n/d) v_\pi(U_d). 
$$

Now, assume the characteristic $p > 0$. 
suppose that $\pi$ is a prime divisor of $U_n$ which is not primitive, where $p \nmid n$. 
Let $m$ be the minimal positive integer such that $\pi \mid U_m$. Automatically, $p \nmid m$. 
Then, by Theorem~\ref{thm:strong3} we have  $m \mid n$, and by Lemma~\ref{lem:vU}, for any positive integer $k$ with $p \nmid k$ 
$$
v_\pi(U_{mk}) = v_\pi(U_m).
$$
Hence, if $m < n$, noticing $p \nmid n$ we obtain 
\begin{align*}
v_\pi(Q_n(\lambda, \eta)) & = \sum_{d \mid n/m} \mu(n/(dm)) v_\pi(Q_{dm}) \\
& = \sum_{d \mid n/m} \mu(n/(dm)) v_\pi(Q_m) \\
& = v_\pi(Q_m) \sum_{d \mid n/m} \mu(n/(dm)) = 0.
\end{align*}
So, any non-primitive prime divisor of $U_n$ (in the sequence $U^\prime$) does not divide $Q_n(\lambda, \eta)$. 
It is easy to see that when $n > 2$, $Q_n(\lambda, \eta) = \Phi_n(\lambda, \eta)$ is non-constant  
(because at least one of $\lambda$ and $\eta$ is transcendental over $K$),  
and so $Q_n(\lambda, \eta) $ has a prime divisor in $R$. 
Thus, when $n > 2$, any prime divisor of $Q_n(\lambda, \eta)$ is primitive, 
and so each term in the sequence $U^\prime$ beyond the second has a primitive prime divisor. 

The proof for the case $p = 0$ follows exactly the same way.  
\end{proof}

\begin{remark}   \label{rem:Lehmer}
In the proof of Theorem~\ref{thm:primitive3}, we obtain more: the \textit{primitive part} 
(that is, the product of all the primitive prime divisors to their respective powers) 
of $U_n$ is $Q_n(\lambda, \eta) = \Phi_n(\lambda,\eta)$, where $n \ge 3$, and $p \nmid n$ if $p > 0$. 
\end{remark}

\section{Proofs of Theorems~\ref{thm:strong2} and \ref{thm:primitive2}}

The proofs follow easily from Theorems~\ref{thm:strong3} and \ref{thm:primitive3}. 

\begin{proof}[Proof of Theorem~\ref{thm:strong2}]
Fix positive integers $m, n$ with $d= \gcd(m,n)$. 
If either both $m,n$ are odd, or both $m,n$ are even, 
it follows directly from Theorem~\ref{thm:strong3} that $\gcd(L_m, L_n) = L_d$ (setting $\lambda = \alpha, \eta = \beta$). 

Now, without loss of generality, assume that $m$ is even and $n$ is odd. By Theorem~\ref{thm:strong3}, we have 
$$
\gcd\Big(\frac{\alpha^m - \beta^m}{\alpha^2 - \beta^2}, \frac{\alpha^n - \beta^n}{\alpha - \beta} \Big) 
= \frac{\alpha^d - \beta^d}{\alpha - \beta}. 
$$
Using Lemma~\ref{lem:abn} we know that $(\alpha^n - \beta^n)/(\alpha - \beta)$ and $\alpha + \beta$ are coprime in $R$. 
Hence, we obtain $\gcd(L_m, L_n) = L_d$. 
This completes the proof. 
\end{proof}

\begin{proof}[Proof of Theorem~\ref{thm:primitive2}]
Assume that the characteristic $p = 0$.  
First, by Lemma~\ref{lem:abn}, we have that for any odd $n \ge 3$, $\Phi_n(\alpha,\beta)$ and $\alpha + \beta$ are coprime in $R$. 

Now, fix an even integer $n \ge 4$. 
Suppose that there exists an irreducible polynomial, say $\pi$, in $R$ dividing both $\Phi_n(\alpha, \beta)$ and $\Phi_2(\alpha,\beta) = \alpha + \beta$.  
This means that the polynomial $X^n - Y^n$, defined over the fraction field of the ring $R$ (mod $\pi$), has a multiple root (that is, $(\alpha, \beta)$). 
However, this fraction field has characteristic zero (because it contains the field $K$), which implies that $X^n - Y^n$ is in fact a simple polynomial. 
Hence, this leads to a contradiction, and so $\Phi_n(\alpha, \beta)$ and $\alpha + \beta$ are coprime in $R$. 

Therefore, by constructions we directly obtain from Remark~\ref{rem:Lehmer} that 
the primitive part of $L_n$ is $\Phi_n(\alpha,\beta)$, where $n \ge 3$. 

Finally, if the characteristic $p > 0$, then by contruction the above arguments still work
(because in the sequence $L^\prime$ we have deleted those terms $L_n$ with $p \mid n$).
\end{proof}

\begin{remark}   \label{rem:Lucas}
In the proof of Theorem~\ref{thm:primitive2}, we obtain more: the primitive part
of $L_n$ is $\Phi_n(\alpha,\beta)$, where $n \ge 3$, and $p \nmid n$ if $p > 0$. 
\end{remark}

\section{Proofs of Theorems~\ref{thm:strong1} and \ref{thm:primitive1}}

Clearly, Theorem~\ref{thm:strong1} follows directly from Theorem~\ref{thm:strong2}. 

\begin{proof}[Proof of Theorem~\ref{thm:primitive1}]
Assume that the characteristic $p = 0$.  
Fix an integer $n \ge 3$. 
Taking $\alpha = f$ and $\beta = g$ in Theorem~\ref{thm:primitive2} and noticing Remark~\ref{rem:Lucas}, 
we know that the primitive part of the term $ (f^n - g^n) / (f-g)$ is $\Phi_n(f,g)$. 
As the above, we obtain that $\Phi_n(f,g)$ and $f-g$ are coprime in $R$. 
Hence, the primitive part of the term $F_n = f^n - g^n$ is $\Phi_n(f,g)$. 

Finally, if the characteristic $p > 0$, then by contruction the above arguments still work 
(because in the sequence $F^\prime$ we have deleted those terms $F_n$ with $p \mid n$).
\end{proof}

\begin{remark}
In the proof of Theorem~\ref{thm:primitive1}, we obtain more: the primitive part
of $F_n$ is $\Phi_n(f,g)$, where $n \ge 3$, and $p \nmid n$ if $p > 0$. 
\end{remark}

\section{Comments}

In this section, we make some remarks about extending our results to unique factorization domains. 

Note that all the lemmas used in proving the strong divisibility property are valid for any unique factorization domain. 
So, we have the following result. 

\begin{theorem}
The strong divisibility properties in Theorems~\ref{thm:strong3}, \ref{thm:strong2} and \ref{thm:strong1} still hold 
when we replace the ring $R$ by a unique factorization domain $D$. 
\end{theorem}

In order to extend fully all our results on primitive divisors to a unique factorization domain $D$, we need to assure two properties. 
One is about the valuation similar as in Lemma~\ref{lem:vU}. 
The other is to assure that $\Phi_n(f,g), \Phi_n(\alpha,\beta)$ and $\Phi_n(\lambda, \eta)$ are all non-zero and non-unit whenever $n \ge 3$. 

If $D$ contains a field, then any integer as an element in $D$ is either zero or a unit, 
and so the valuation result holds in this case by following the same arguments as Lemma~\ref{lem:vU}. 
Hence, in this case, if one can show that $\Phi_n(f,g)$ is non-unit whenever $n > n_0$ for some integer $n_0$, 
then one in fact prove the result in Theorem~\ref{thm:primitive1} by replacing ``beyond the second" with ``beyond the $n_0$-th". 
Similar things apply to Theorems~\ref{thm:primitive3} and \ref{thm:primitive2}. 

We present an example here. 
Let $D=K[[X]]$ be the formal power series ring defined over a field $K$ in one variable $X$. 
Then, an element $\sum_{n=0}^{\infty}a_n X^n$ in $D$ is a unit  if and only if $a_0 \ne 0$.  
Let $f$ and $g$ be non-zero, non-unit and coprime in $D$ such that $f/g$ is not a root of unity. 
Then, $\Phi_n(f,g)$ is non-zero and non-unit for any $n \ge 1$, 
and so Theorem~\ref{thm:primitive1} holds in this case. 
In addition, if let $f$ be non-unit and $g$ a unit in $D$, then $\Phi_n(f,g)$ is a unit for any $n \ge 1$,

\section*{Acknowledgement}
The author was partly supported by the Australian Research Council Grant DE190100888.

\end{document}